\documentclass[11pt]{amsart}
\usepackage{pdfsync}
\usepackage{fullpage}
\usepackage{amsmath}
\usepackage{amssymb}
\usepackage{MnSymbol}
\usepackage{verbatim}
\usepackage[linktocpage]{hyperref}

\usepackage{mathtools}
\newtheorem{df}{Definition}[section]
\newtheorem{thm}[df]{Theorem}

\newtheorem{rem}[df]{Remark}

\newtheorem{lem}[df]{Lemma}

\newtheorem{cor}[df]{Corollary}
\newtheorem{conj}[df]{Conjecture}

\newcommand{\eq}{eqnarray*}

\newcommand{\ti}{\tilde}

\title{Harmonic maps for Hitchin representations}
\author{Qiongling Li}
\address{
Centre for Quantum Geometry of Moduli Spaces (QGM)\\
Aarhus University\\
Ny Munkegade 118\\
8000 Aarhus C, Denmark}
\address{
Department of Mathematics\\
California Institute of Technology\\
1200 East California Boulevard\\
 Pasadena, US}

\email{qiongling.li@gmail.com}

\begin{document}
\maketitle
\begin{abstract}
Let $(S,g_0)$ be a hyperbolic surface, $\rho$ be a Hitchin representation for $PSL(n,\mathbb R)$, and $f$ be the unique $\rho$-equivariant harmonic map from $(\widetilde S, \widetilde g_0)$ to the corresponding symmetric space. We show its energy density satisfies $e(f)\geq 1$ and equality holds at one point only if $e(f)\equiv 1$ and $\rho$ is the base $n$-Fuchsian representation of $(S,g_0)$. In particular, we show given a Hitchin representation $\rho$ for $PSL(n,\mathbb R)$, every $\rho$-equivariant minimal immersion $f$ from a hyperbolic plane $\mathbb H^2$ into the corresponding symmetric space $X$ is distance-increasing, i.e. $f^*(g_{X})\geq g_{\mathbb H^2}$. Equality holds at one point only if it holds everywhere and $\rho$ is an $n$-Fuchsian representation.
\end{abstract}
\section{Introduction}
We study equivariant harmonic maps into the noncompact symmetric space for Hitchin representations. Consider a closed orientable surface $S$ of genus $g\geq 2$. By the uniformization theorem, a Riemann surface structure, or equivalently a hyperbolic metric on $S$ gives rise to a Fuchsian holonomy of $\pi_1(S)$ into $PSL(2,\mathbb R)$. Composing it with the unique irreducible representation embedding of $PSL(2,\mathbb R)$ into $PSL(n,\mathbb R)$, we get the base $n$-Fuchsian representation into $PSL(n,\mathbb R)$ of the Riemann surface. The Hitchin representations are exactly deformations of the $n$-Fuchsian representations. Those form a connected component inside the representation variety for $PSL(n,\mathbb R)$, called the Hitchin component. Hitchin representations possess lots of nice properties, for example: Hitchin \cite{Hitchin92} showed that they are irreducible; Labourie \cite{LabourieAnosov} showed that they are discrete, faithful quasi-isometric embeddings.

Given $M, N$ two Riemannian manifolds and consider an equivariant map $f:\widetilde M\rightarrow N$ for a representation $\rho:\pi_1(M)\rightarrow Isom(N)$. Recall the energy density of $f$  is $e(f)=\frac{1}{2}||df||^2$ with respect to the metrics on $\widetilde M$ and $N$. By equivariance, $e(f)$ is invariant under $\pi_1(M)$, so it gives a well-defined function on $M$, also called the energy density. The energy $E(f)$ is the integral of the energy density $e(f)$ over $M$. Equivariant harmonic maps are critical points of the energy functional. 

Fix a Riemann surface structure $\Sigma$ over $S$. Following the work of Donaldson \cite{Donaldson} and Corlette \cite{Corlette}, for any irreducible representation $\rho$ into a semisimple Lie group $G$, there exists a unique $\rho$-equivariant harmonic map $f$ from $\widetilde\Sigma$ to the corresponding symmetric space of $G$. In the base $n$-Fuchsian representation of $\Sigma$, the harmonic map $f$ is a totally geodesic embedding of $\mathbb H^2$. The equivariant harmonic map further gives rise to a Higgs bundle, a pair $(E,\phi)$ consisting of a holomorphic vector bundle over $\Sigma$ and a holomorphic section of $End(E)\otimes K$, the Higgs field. Conversely, by the work of Hitchin \cite{Hitchin92} and Simpson \cite{Simpson}, a stable Higgs bundle admits a unique harmonic metric on the bundle solving the Hitchin equation. The harmonic metric further gives rise to an irreducible representation $\rho$ into $G$ and a $\rho$-equivariant harmonic map into the corresponding symmetric space. These two directions between representations and Higgs bundles together give the celebrated non-abelian Hodge correspondence. In particular, Hitchin component has a nice Higgs bundle description by Hitchin \cite{Hitchin92}.


We study the energy density of equivariant harmonic maps for Hitchin representations. Let $X$ denote the symmetric space $SL(n,\mathbb R)/SO(n)$. Renormalize the metric on $X$ such that in the base $n$-Fuchsian case, the totally geodesic copy of the hyperbolic plane inside $X$ is also of curvature $-1$. Here we prove
\begin{thm}\label{maintheorem}(Theorem \ref{Maintheorem})
Let $\rho$ be a Hitchin representation for $PSL(n,\mathbb R)$, $g_0$ be a hyperbolic metric on $S$, and $f$ be the unique $\rho$-equivariant harmonic map from $(\widetilde S, \widetilde g_0)$ to the symmetric space $X$. Then its energy density $e(f)$ satisfies
$$e(f)\geq 1.$$
Moreover, equality holds at one point only if $e(f)\equiv 1$ in which case $\rho$ is the base $n$-Fuchsian representation of $(S,g_0)$.  
\end{thm}

\begin{rem}
Consider the $(1,1)$-part of a Riemannian metric $g$, that is, $g^{1,1}=2Re(g(\partial_{z},\partial_{\bar z})dz\otimes d\bar z)$.  Equivalently, Theorem \ref{maintheorem} says the pullback metric of $f$ satisfies $(f^*g_{X})^{1,1}\geq \widetilde g_0$ as the domination of two $2$-tensors, meaning their difference is a definite $2$-tensor.
\end{rem}

\begin{rem} In the case $n=2$, it concerns harmonic diffeomorphism between surfaces and is proven in Sampson \cite{Sampson}. For $n\geq 3$, the author and Dai in \cite{DL} verified the cyclic families. \end{rem}

Understanding the harmonic maps involves estimating the harmonic metric solving the Hitchin equation for Higgs bundles. The Hitchin equation is a second-order nonlinear elliptic system, which is highly non-trivial in general. The property of cyclic Higgs bundles used in \cite{DL} is that the solution metric is diagonal with respect to the holomorphic splitting, which is also essential in Labourie's proof \cite{LabourieCyclic} showing the uniqueness of equivariant minimal surface for Hitchin representations into split real rank $2$ Lie group.

However, the diagonal property no longer holds for general cases and it seems hard to directly analyze the Hitchin equation for a particular Higgs bundle. Here we are able to analyze the Hitchin equation without requiring cyclic condition by choosing another orthogonal splitting of the bundle $E$. It suggests the new splitting could be a more effective tool to study Hitchin representations and related harmonic maps.\\

Recall the energy of $f$ is the integral of the energy density on $S$ with respect to the hyperbolic volume form. As a corollary of Theorem \ref{maintheorem}, we reprove the following result of Hitchin \cite{Hitchin92}.
\begin{cor}\label{energy} Under the same assumptions of Theorem \ref{maintheorem}, the energy $E(f)$ of every $\rho$-equivariant harmonic map $f$ satisfies 
$$E(f)\geq -2\pi\cdot \chi(S).$$
Equality holds if and only if $\rho$ is the base $n$-Fuchsian representation of $(S,g_0)$.
\end{cor}
\begin{rem}
The energy is also the $L^2$-norm of the Higgs field, which is a Morse-Bott function on the moduli space of Higgs bundles. It was first considered by Hitchin \cite{Hitchin87, Hitchin92} to determine the topology of the moduli space and hence the representation variety. Hitchin's original proof of Corollary \ref{energy} in Higgs bundle language relies on the K\"ahler structure on the moduli space and that the energy function is a moment map of a $S^1$-action on the moduli space.
\end{rem}

When the harmonic map is conformal, it is a minimal immersion. Labourie \cite{LabourieEnergy} showed that for any Hitchin representation $\rho$, there exists a $\rho$-equivariant immersed minimal surface inside the symmetric space $X$. 

As a special case of Theorem \ref{maintheorem}, we show the metric domination theorem, which was conjectured by Dai and the author in \cite{DL}.
\begin{thm}\label{secondmaintheorem}\label{DominationConjecture}
Let $\rho$ be a Hitchin representation for $PSL(n,\mathbb R)$. Then the pullback metric of every $\rho$-equivariant minimal immersion $f$ of the hyperbolic plane $\mathbb H^2$ into the symmetric space $X$ satisfies
\begin{\eq}
f^*(g_{X})\geq g_{\mathbb H^2},
\end{\eq}
Equality holds at one point if and only if it holds everywhere in which case $\rho$ is an $n$-Fuchsian representation. 
\end{thm}

\begin{rem}
For general representation which are not Hitchin, one should not expect similar phenomenon happens. For example, for any reductive $SL(2,\mathbb C)$-representation $\rho$, every $\rho$-equivariant minimal mapping from $\mathbb H^2$ into $\mathbb H^3$ is distance-decreasing by Ahlfors' lemma (see \cite{DominationFuchsian}). Even for split real groups, there exist many maximal $Sp(4,\mathbb R)$-representations and equivariant minimal immersions from $\mathbb H^2$ into $Sp(4,\mathbb R)/U(2)$ are distance-decreasing (see \cite{DL2}).
\end{rem}

We equip $\rho$ with a translation length spectrum $\bar l_{\rho}:\gamma\in \pi_1(S)\rightarrow\mathbb R^+$ given by $$\bar l_{\rho}(\gamma):=\inf_{M}\inf_{x\in M}d_M(x,\rho(\gamma)(x)),$$ where $M$ goes through all $\rho$-equivariant minimal surfaces in the symmetric space $X$ and $d_M$ is induced by the metric on the minimal surface $M$. One may compare $\bar l_{\rho}$ with the classical translation length spectrum $ l_{\rho}:\gamma\in \pi_1(S)\rightarrow\mathbb R^+$ given by $$l_{\rho}(\gamma):=\inf\limits_{x\in X}d_{X}(x,\rho(\gamma)(x)),$$ where $d_{X}$ is induced by the unique $G$-invariant Riemannian metric on $X$. Notice that for an $n$-Fuchsian representation $j$, $\bar l_j=l_j$ and the fact $d_{X}\leq d_M$ implies that $l_{\rho}\leq \bar l_{\rho}$. As a direct corollary of Theorem \ref{secondmaintheorem}, we have 

\begin{cor} For any Hitchin representation $\rho$ for $PSL(n,\mathbb R)$, then either it is an $n$-Fuchsian representation or there exists an $n$-Fuchsian representation $j$ such that $\bar l_{\rho}>\bar l_j=l_j$. \end{cor}

\subsection*{Organization of the paper} In Section \ref{pre}, we review the theory of Higgs bundles and the non-abelian Hodge correspondence. In Section \ref{Hitchin}, we first explain Higgs bundle description of Hitchin representations in terms of holomorphic differentials, then introduce a new expression of Higgs bundle in Hitchin component, and in the end deduce the Hitchin equation using the new expression. In Section \ref{proof}, we show the main theorem. In Section \ref{question}, we discuss some further questions.

\subsection*{Acknowledgement} The author is supported in part by the center of excellence grant `Center for Quantum Geometry of Moduli Spaces' from the Danish National Research Foundation (DNRF95). The author acknowledges support from U.S. National Science Foundation grants DMS 1107452, 1107263, 1107367 ``RNMS: GEometric structures And Representation varieties" (the GEAR Network).

\section{Preliminaries}\label{pre}
In this section, we briefly recall the theory of Higgs bundles and the non-abelian Hodge correspondence. One may refer \cite{Bar}\cite{DL}\cite{LabourieCyclic} for more details. For $p\in S$, let $\pi_1=\pi_1(S,p)$ be the fundamental group of $S$. Let $\Sigma$ be a Riemann surface structure on $S$, $\tilde{\Sigma}$ be the universal cover, and $K_{\Sigma}$ be the canonical line bundle over $\Sigma$. Let $g_0$ be the conformal hyperbolic metric on $\Sigma$ of constant curvature $-1$.

Through this whole section, $G$ denotes $SL(n,\mathbb C)$ and $K$ denotes $SU(n)$.

\subsection{From Higgs bundles to harmonic maps and representations}
\begin{df}
A $G$-Higgs bundle over $\Sigma$ is a pair $(E,\phi)$ containing a holomorphic rank $n$ vector bundle of trivial determinant and a trace-free holomorphic bundle map from $E$ to $E\otimes K_{\Sigma}$. 

We call $(E,\phi)$ stable if any proper $\phi$-invariant holomorphic subbundle $F$ has negative degree, and polystable if it is a direct sum of stable Higgs bundles of degree $0$. The moduli space $\mathcal{M}_{Higgs}(G)$ is a space of gauge equivalent classes of polystable $G$-Higgs bundles. 
\end{df}

\begin{thm}(Hitchin \cite{Hitchin87} and Simpson \cite{Simpson})\label{dh}
Let $(E,\phi)$ be a stable $G$-Higgs bundle. Then there exists a unique Hermitian metric $H$ on $E$ compatible with $G$-structure, called harmonic metric, solving the Hitchin equation
\begin{eqnarray*}
F^{\nabla^H}+[\phi,\phi^{*_H}]=0, 
\end{eqnarray*}
where ${\nabla^H}$ is the Chern connection of $H$
and $\phi^{*_{H}}$ is the adjoint of $\phi$ with respect to $H$.
\end{thm}

The Hitchin equation is equivalent to the $G$-connection $D=\nabla^{H}+\phi+\phi^{*_{H}}$ being flat. The holonomy of $D$ gives a reductive representation $\rho:\pi_1\to G$ and the pair $(E,D)$ is isomorphic to $\widetilde{\Sigma}\times_{\rho}\mathbb{C}^n$ equipped with the natural flat connection.  

A choice of a Hermitian metric $H$ on $E$ is equivalent to a reduction of the unimodule frame bundle $P_G$ of $E=\widetilde{\Sigma}\times_{\rho}\mathbb{C}^n$ to $K$ by considering the unitary frame bundle. It then descends to be a section of $P_G/K=\widetilde{\Sigma}\times_{\rho}G/K$ over $\Sigma$. Equivalently, it corresponds to a $\rho$-equivariant map $f: \tilde{\Sigma}\to G/K$. Moreover, the map arising from a Hermitian metric solving the Hitchin equation is harmonic. If the holonomy of the representation lies in $SL(n,\mathbb R)$, the harmonic map lies in a totally geodesic copy of $SL(n,\mathbb R)/SO(n)$ inside $G/K=SL(n,\mathbb C)/SU(n)$.

\subsection{From representations and harmonic maps to Higgs bundles}

Denote $Rep(\pi_1,G)$ the set of conjugate classes of reductive representations of $\pi_1(S)$ into $G$. Given a reductive representation, by the work of Corlette \cite{Corlette} and Donaldson \cite{Donaldson}, there exists a $\rho$-equivariant harmonic map $f:\widetilde{\Sigma}\rightarrow G/K$, which is unique up to the centralizer of $\rho(\pi_1)$. 

Let's explain how to extract a Higgs bundle from an equivariant harmonic map. Denote \begin{\eq}
\mathfrak{g}=sl(n,\mathbb{C}),\quad \mathfrak{k}=su(n),\quad \mathfrak{p}=\{X\in sl(n,\mathbb{C}):\bar X^t=X\}.
\end{\eq} 

With respect to the $Ad(K)$-invariant decomposition $\mathfrak{g}=\mathfrak{k}+\mathfrak{p}$, we can decompose the Maurer-Cartan form of $G$, $\omega=\omega^{\mathfrak{k}}+\omega^{\mathfrak{p}}$, where $\omega^{\mathfrak{k}}\in \Omega^1(G,\mathfrak{k}), \omega^{\mathfrak{p}}\in \Omega^1(G,\mathfrak{p})$. $\omega^{\mathfrak k}$ is a connection on the $K$-bundle $G\rightarrow G/K$ and $\omega^{\mathfrak{p}}$ descends to be an element in $\Omega^1(G/K,G\times_{Ad_K}\mathfrak{p})$, giving an isomorphism $T(G/K)\cong G\times_{Ad_K}\mathfrak{p}$. Pulling back the $K$-bundle $G\rightarrow G/K$ to $\widetilde\Sigma$ by $f$, we get a principal $K$-bundle $P_K$ on $\widetilde\Sigma$. Then\\
(1) $f^*\omega^{\mathfrak k}$ is a connection form on $P_K$, hence a unitary connection form $A$ on the complexified bundle $P_G$. Denote the covariant derivative of $A$ as $d_A$.\\
(2) $f^*\omega^{\mathfrak{p}}$ is a section of $T^*\tilde{\Sigma}\otimes (\tilde{P}_K\times_{Ad_K}\mathfrak{p})$ over $\tilde{\Sigma}$, whose complexification is
\begin{\eq}
(T^*{\Sigma}\otimes\mathbb{C})\otimes ({P}_{K}\times_{Ad_{K}}\mathfrak{p}\otimes\mathbb{C})
&=&(K\oplus\bar{K})\otimes ({P}_G\times_{Ad_G}\mathfrak{g})\\&=&(K\oplus\bar{K})\otimes End_0(E)\end{\eq}where $End_0(E)$ is the trace-free endomorphism bundle of $E$.

The harmonicity of the map $f$ assures that the pair $(d_A^{0,1}, (f^*\omega^{\mathfrak{p}})^{1,0})$ is a $G$-Higgs bundles over $\widetilde\Sigma$. By the equivariance of $f$, the Higgs bundle descends to $\Sigma$. Since the flat connection on $P_G$ is from the Maurer-Cartan form $\omega$, by comparing these two decomposition, the harmonic map associated to this Higgs bundle coincides with the original one. 

In this way, we obtain a homeomorphism $$\mathcal{M}_{Higgs}(G)\cong Rep(\pi_1, G),$$ i.e. the non-abelian Hodge correspondence. 

\subsection{Harmonic maps in terms of Higgs field}
By the above discussion, with the form $\omega^{\mathfrak p}\in \Omega^1(G/K, G\times_{Ad_K}\mathfrak{p})$, the conclusion is that: \begin{eqnarray*}(f^*\omega^{\mathfrak{p}})^{1,0}=\pi^*\phi\in \Omega^{1,0}(\widetilde\Sigma, End_0(E)),\\
 (f^*\omega^{\mathfrak{p}})^{0,1}=\pi^*\phi^{*_H}\in \Omega^{0,1}(\widetilde\Sigma, End_0(E)).\end{eqnarray*}
The rescaled Killing form $B(X,Y)=\frac{12}{n(n^2-1)}\cdot\text{tr}(XY)$ on $\mathfrak{g}$  induces a Riemannian metric $\tilde{B}$ on $G/K$: for two vectors $Y_1,Y_2\in T_p(G/K)$, $$\ti B(Y_1,Y_2)=B(\omega^{\mathfrak{p}}(Y_1),\omega^{\mathfrak{p}}(Y_2)).$$ We will later see in Equation (\ref{Fuchsian}) that this rescaled Killing form is such that in the base $n$-Fuchsian case, the equivariant harmonic map is a totally geodesic isometric embedding of $(\widetilde S,\widetilde g_0)$ inside $G/K$. Pulling back the metric to $\widetilde\Sigma$: for any two vectors $\tilde{X},\tilde{Y}\in T\tilde{\Sigma}$,
\begin{\eq}
f^*g_{G/K}(\tilde X,\tilde Y)=\tilde{B}(f_*(\tilde{X}),f_*(\tilde{Y}))=B(\omega^{\mathfrak{p}}(f_*(\tilde{X})),\omega^{\mathfrak{p}}(f_*(\tilde{Y})).
\end{\eq}

Since $f$ is $\rho$-equivariant and $\tilde{B}$ is $G$-invariant, $f^*g_{G/K}$ also descends to $\Sigma$. From now on, we won't distinguish notations on $\Sigma$ and $\widetilde \Sigma$. So the pullback metric $f^*g_{G/K}$ on $\Sigma$ is $$\frac{12}{n(n^2-1)}(\text{tr}(\phi^2)+2 Re(\text{tr}(\phi\phi^{*_H}))+\text{tr}(\phi^{*_H}\phi^{*_H})).$$ The Hopf differential of $f$ is$$\text{Hopf}(f)=g_f^{2,0}=\frac{12}{n(n^2-1)}\cdot\text{tr}(\phi^2).$$

Now let us derive the exact formula of the energy density $e(f)=\frac{1}{2}||df||^2,$ the square norm of $df$ as a section of $T^*\tilde S\otimes f^{-1}T(G/K)$. The hyperbolic metric $g_0$ on $S$ also gives a Hermitian metric $h$ on $K_{\Sigma}^{-1}$ satisfying $g=2Re(h)$. In local coordinate $z$, denote the local function $\hat g_0=g_0(\partial_{z},\partial_{\bar z})=h(\partial_{z},\partial_{z})$. And $\hat g_0$ satisfies $$\partial_{\bar z}\partial_z\log \hat g_0=\hat g_0,$$ since the curvature formula is $K_{g_0}=-\frac{1}{\hat g_0}\partial_{\bar z}\partial_z\log \hat g_0$, independent of choices of coordinates. Locally, 
$$g_0=2 Re (\hat g_0 dz\otimes d\bar z)=\hat g_0(dz\otimes d\bar z+d\bar z\otimes dz)=2\hat g_0(dx^2+dy^2).$$

Locally, the energy density \begin{eqnarray*}
e(f)&=&\frac{1}{2}||df||^2=\frac{1}{4\hat g_0}(\tilde B(f_*(\partial_x),f_*(\partial_x))+\tilde B(f_*(\partial_y),f_*(\partial_y)))\\
&&\text{Set $\Phi=\phi+\phi^{*_H}$.}\\
&=&\frac{1}{4\hat g_0}(B(\Phi(\partial_x),\Phi(\partial_x))+B(\Phi(\partial_y),\Phi(\partial_y)))\\
&&\text{Denote $\phi=\hat \phi dz$ and $\phi^{*_H}=\hat\phi^{*_H}d\bar z$.}\\
&=&\frac{1}{4\hat g_0}(B(\hat\phi+\hat\phi^{*_H},\hat\phi+\hat\phi^{*_H})+B(i(\hat\phi-\hat\phi^{*_H}),i(\hat\phi-\hat\phi^{*_H}))\\
&=&\frac{1}{\hat g_0}B(\hat\phi,\hat\phi^{*_H})=\frac{1}{\hat g_0}\text{tr}(\hat\phi\hat\phi^{*_H}).
\end{eqnarray*} Hence globally,
$$2Re(\text{tr}(\phi\phi^{*_H}))=e(f)\cdot g_0.$$

\section{Higgs bundles in Hitchin component}\label{Hitchin}
In the first part of this section, we explain Higgs bundle description of Hitchin representations in terms of holomorphic differentials $\bigoplus\limits_{i=2}^n H^0(K^i)$. In the second part, we introduce a new expression of Higgs bundle in Hitchin component behaving nicely with respect to the harmonic metric. Lastly, we deduce the Hitchin equation using the new Higgs bundle expression. 

\subsection{Hitchin fibration and Hitchin section} Let us explain Hitchin's parametrization in \cite{Hitchin92} and refer to \cite{Bar} for details. Choose any principal $3$-dimensional subalgebra $\mathfrak s=span\{\tilde e_1,x,e_1\}$ in $\mathfrak g=sl(n,\mathbb C)$ consisting of a semisimple element $x$, regular nilpotent elements $\tilde e_1$ and $e_1$ satisfying $$[x,e_1]=e_1,\quad [x,\tilde e_1]=-\tilde e_1,\quad  [e_1,\tilde e_1]=x.$$  Moreover, we require $\mathfrak s$ to be real with respect to a compact real form $\rho$ of $\mathfrak g$, that is, 
$$\rho(x)=-x,\quad \rho(e_1)=-\tilde e_1.$$ The adjoint representation of $\mathfrak s$ decomposes $\mathfrak g$ into a direct sum of irreducible representations $V_i$ of dimension $2i+1$ for $1\leq i\leq n-1$ and $V_1$ is just $\mathfrak s$ itself. Let $e_i\in V_i$ is the eigenvector of $ad(x)$ of the highest weight, which is $i$.

Consider now elements $f=\tilde e_1+\alpha_2 e_1+\cdots+\alpha_n e_{n-1}\in \mathfrak g$, there exist invariant polynomials $p_i, i=2,\cdots,n$ of degree $i$ on $\mathfrak g$ such that $p_i(f)=\alpha_i$. The Hitchin fibration is then defined as a map from the moduli space of $SL(n,\mathbb{C})$-Higgs bundles over $\Sigma$ to the direct sum of the holomorphic differentials
\begin{\eq}
h:M_{Higgs} &\longrightarrow&\bigoplus\limits_{j=2}^nH^0(\Sigma, K^j)\ni (q_2,q_3,\cdots,q_n)\\
(E,\phi)&\longmapsto& (p_2(\phi),p_3(\phi),\cdots,p_n(\phi)).\end{\eq} 

We can define a section of the Hitchin fibration as follows: for a tuple $(q_2,\cdots,q_n)\in \bigoplus\limits_{j=2}^nH^0(\Sigma, K^j)$, define its section as $$E=K^{\frac{n-1}{2}}\oplus K^{\frac{n-3}{2}}\oplus\cdots\oplus K^{\frac{1-n}{2}},\quad \phi=\tilde e_1+q_2e_1+q_3e_2+\cdots+q_ne_n,$$ where $(q_2,q_3,\cdots,q_n)\in\bigoplus\limits_{i=2}^n H^0(K^i)$. By Htichin \cite{Hitchin92}, such Higgs bundles are stable and have holonomy inside a copy of $SL(n,\mathbb R)$ inside $SL(n,\mathbb C)$. Under the non-abelian Hodge correspondence, the Hitchins section correponds to the Hitchin component in the representation variety for $PSL(n,\mathbb R)$. Different choices of principal $3$-dimensional subalgebras give gauge equivalent Higgs bundles.  \\

Now we will write the Higgs bundles in Hitchin component in matrix form explicitly. 

Choose the compact real form on $sl(n,\mathbb C)$ as $\rho(X)=-\overline{X}^T$. Choose the Cartan subalgebra as the trace-free diagonal matrices $\mathfrak{h}\ni H=diag(t_1,\cdots,t_n)$ and a positive Weyl chamber in $\mathfrak t$ consisting of $H$ satisfying $t_i>t_{i+1}$. The root space $\triangle$, positive root space $\triangle^+$, and simple root space $\Pi$ are 
\begin{eqnarray*}&&\triangle=\{\alpha_{ij}\in \mathfrak h^*, i\neq j| \alpha_{ij}(H)=t_i-t_j, H\in \mathfrak h\},\\
&&\triangle^+=\{\alpha_{ij}\in\triangle, i>j\},\quad \Pi=\{\alpha_{i,i+1}\in \triangle\}.\end{eqnarray*}

Let $E_{ij}$ be a $n\times n$ matrix such that its $(i,j)$-entry is $1$ and $0$ elsewhere. Denote by $<,>$ the killing form of $\mathfrak g$ and the same symbol for its restriction to $\mathfrak h$ and also its dual extension to $\mathfrak h^*$. For each root $\alpha\in \triangle,$ the coroot $h_{\alpha}\in \mathfrak h$ is defined as $<h_{\alpha},u>=\frac{2}{<\alpha,\alpha>}\alpha(u)$ for $u\in \mathfrak h$. So the coroot of $\alpha_{ij}$ is $h_{\alpha_{ij}}=E_{ii}-E_{jj}$, and its root space is spanned by $x_{\alpha_{ij}}=E_{ij}$. We then choose $\tilde e_1,x,e_1$ as follows,
\begin{eqnarray*}
&&x=\frac{1}{2}\sum_{\alpha\in \triangle^{+}} h_{\alpha}=\sum_{\alpha_{i,i+1}\in \Pi}r_ih_{\alpha_{i,i+1}}=\begin{pmatrix}
\frac{n-1}{2}&&&&\\
&\frac{n-3}{2}&&&\\
&&\ddots&&\\
&&&\frac{3-n}{2}&\\
&&&&\frac{1-n}{2}
\end{pmatrix},\\
&&e_1=\sum_{\alpha_{i,i+1}\in\Pi}\sqrt{r_i}x_{\alpha_{i,i+1}}=\begin{pmatrix}
0&\sqrt{r_1}&&&\\
&0&\sqrt{r_2}&&\\
&&\ddots&\ddots&\\
&&&0&\sqrt{r_{n-1}}\\
&&&&0\end{pmatrix},\\
&&\tilde e_1=\sum_{\alpha_{i,i+1}\in \Pi}\sqrt{r_i}x_{-\alpha_{i,i+1}}=\begin{pmatrix}
0&&&&\\
\sqrt{r_1}&0&&&\\
&\sqrt{r_2}&0&&\\
&&\ddots&\ddots&\\
&&&\sqrt{r_{n-1}}&0\end{pmatrix},
\end{eqnarray*}
where $r_i=\frac{i(n-i)}{2}$. 

One can check that $\mathfrak s=span\{\tilde e_1,x,e_1\}$ is a principal $3$-dimensional subalgebra which is real with respect to the compact real form $\rho$.
 As a generalization of $e_1$, the vector $$e_i=\sum_{k=1}^{n-i}(\prod_{j=k}^{k+i-1}\sqrt{r_j})\cdot E_{k,k+i}$$ is the eigenvector of $ad(x)$ of eigenvalue $i$ and satisfies $[e_1,e_i]=0$. The vector space generated by $\mathfrak s$ and $e_i$ is of dimension $2i+1$ in which $e_i$ is the eigenvector of highest weight. 

With respect to $\mathfrak s$, the Higgs bundle in Hitchin component corresponding to $(q_2,\cdots,q_n)$ is of the form: $E=K^{\frac{n-1}{2}}\oplus K^{\frac{n-3}{2}}\oplus\cdots\oplus K^{\frac{1-n}{2}}$,
\begin{eqnarray}
&&\phi=\tilde e_1+q_2e_1+q_3e_2+\cdots+q_ne_{n-1}\nonumber\\
&&=\begin{pmatrix}0&\sqrt{r_1}q_2&\sqrt{r_1r_2}q_3&\cdots&\prod_{i=1}^{n-2}\sqrt{r_i}q_{n-1}&\prod_{i=1}^{n-1}\sqrt{r_i}q_n\\\sqrt{r_1}&0&\sqrt{r_2}q_2&\cdots&\cdots&\prod_{i=2}^{n-1}\sqrt{r_i}q_{n-1}\\&\sqrt{r_2}&0&\sqrt{r_3}q_2&\cdots&\vdots\\ &&\ddots&\ddots&\ddots&\vdots\\ &&&0&\sqrt{r_{n-2}}q_2&\sqrt{r_{n-2}r_{n-1}}q_3\\&&&\sqrt{r_{n-2}}&0&\sqrt{r_{n-1}}q_2\\ &&&&\sqrt{r_{n-1}}&0\end{pmatrix}.\nonumber
\end{eqnarray} 

\begin{rem}
In the case $(q_2,\cdots,q_n)=(q_2,0,\cdots,0)$, the corresponding representation is $n$-Fuchsian representation. In the case $(q_2,\cdots,q_n)=(0,\cdots,0)$, it is the base $n$-Fuchsian case of $\Sigma$.
\end{rem}

For our convenience later, we will work with a gauge equivalent Higgs bundle with the above expression under the gauge transformation 
$$g=\begin{pmatrix}
1&&&&\\
&\sqrt{r_1}&&&\\
&&\sqrt{r_1r_2}&&\\
&&&\ddots&\\
&&&&\prod_{i=1}^{n-1}\sqrt{r_i}
\end{pmatrix}.$$
 The holomorphic structure on $E$ does not change and the Higgs field $\phi$ becomes $g^{-1}\phi g$. So the Higgs bundle in Hitchin component is of the following form: $E=K^{\frac{n-1}{2}}\oplus K^{\frac{n-3}{2}}\oplus\cdots\oplus K^{\frac{1-n}{2}},$ and
\begin{equation}\label{expression}
\phi=\begin{pmatrix}0&r_1q_2&r_1r_2q_3&r_1r_2r_3q_4&\cdots&\prod_{i=1}^{n-2}r_iq_{n-1}&\prod_{i=1}^{n-1}r_iq_n\\1&0&r_2q_2&r_2r_3q_3&\cdots&\cdots&\prod_{i=2}^{n-1}r_iq_{n-1}\\&1&0&r_3q_2&r_3r_4q_3&\cdots&\vdots\\ &&\ddots&\ddots&\ddots&\ddots&\vdots\\ &&&1&0&r_{n-2}q_2&r_{n-2}r_{n-1}q_3\\&&&&1&0&r_{n-1}q_2\\ &&&&&1&0\end{pmatrix}，
\end{equation}
where $r_i=\frac{i(n-i)}{2}$.

\subsection{A smooth orthogonal decomposition of the bundle} The Higgs bundle in Hitchin component in the above format is explicit. Since $(E,\phi)$ is stable, there is a unique harmonic metric $H$. The above Higgs bundle expression (\ref{expression}) has a disadvantage that the holomorphic decomposition of $E$ is not an orthogonal decomposition with respect to $H$. To achieve an orthogonal splitting of the bundle, we sacrifice the holomorphic splitting of the bundle and the explicit expression of Higgs bundle in terms of holomorphic differentials. However, we still have some control on the holomorphic structure and Higgs field by carefully choosing the orthogonal splitting of the bundle $E$.

The new decomposition goes as follows. We start with $(E,\phi)$ in the expression $(\ref{expression})$.  The bundle $E$ is a direct sum of holomorphic line bundles $K^{\frac{n-1}{2}}\oplus K^{\frac{n-3}{2}}\oplus\cdots\oplus K^{\frac{1-n}{2}}$. Define $F_k=\bigoplus\limits_{i=1}^k K^{\frac{n+1-2i}{2}}$ to be the direct sum of the first $k$ line bundles. Therefore $E$ admits a holomorphic filtration 
$$0=F_0\subset F_1\subset \cdots F_{n-1}\subset F_n=E.$$ We can equip each line bundle $F_k/F_{k-1}$ with the quotient holomorphic structure. Then\\
(1) the natural inclusion $K^{\frac{n+1-2k}{2}}\subset F_k$ induces an isomorphism between the holomorphic line bundle $K^{\frac{n+1-2k}{2}}$ and $F_k/F_{k-1}$ equipped with the quotient holomorphic structure;\\
(2) the Higgs field $\phi$ takes $F_k$ to $F_{k+1}\otimes K$ and the induced map from $F_k/F_{k-1}\rightarrow F_{k+1}/F_k\otimes K$ is the constant map $1: K^{\frac{n+1-2k}{2}}\rightarrow K^{\frac{n+1-2(k+1)}{2}}\otimes K$ under the isomorphism between $F_k/F_{k-1}$ with $K^{\frac{n+1-2k}{2}}.$\\

For $1\leq k\leq n$, take $L_k$ as the $H$-orthogonal line bundle inside $F_k$ with respect to $F_{k-1}$.  So the inclusion map $L_k\subset F_k$ induces an isomorphism between $L_k$ with the quotient line bundle $F_k/F_{k-1}$. Then $L_k$ is equipped with the pullback quotient holomorphic structure of $F_k/F_{k-1}$. In this way, $E$ has a $C^{\infty}$-decomposition as $L_1\oplus L_2\oplus \cdots \oplus L_n$, where $L_k$ is a holomorphic line bundle which is isomorphic to $K^{\frac{n+1-2k}{2}}$. Note that $L_i$ is not a holomorphic subbundle of $E$ except $k=1$.

Let us record some important information about the decomposition from the construction:\\
(i) the decomposition is orthogonal with respect to $H$;\\
(ii) $\bar\partial _E$ preserves $\bigoplus\limits_{i=1}^k L_i$, since $F_k=\bigoplus\limits_{i=1}^k L_i$ is a holomorphic subbundle of $E$;\\
(iii) $\phi$ take $\bigoplus\limits_{i=1}^k L_i$ to $(\bigoplus\limits_{i=1}^{k+1}L_i)\otimes K$ and the induced map from $$L_k\xrightarrow{\text{$\phi$}} F_{k+1}\otimes K\xrightarrow{\text{pr}_{k+1}} L_{k+1}\otimes K$$ is the constant map $1: K^{\frac{n+1-2k}{2}}\rightarrow K^{\frac{n+1-2(k+1)}{2}}\otimes K$ under the isomorphism between $L_k$ with $K^{\frac{n+1-2k}{2}}.$\\

More explicitly, with respect to the smooth decomposition $$E=L_1\oplus L_2\oplus\cdots\oplus L_n,\quad L_k=K^{\frac{n+1-2k}{2}},$$ we have:

I. the Hermitian metric $H$ solving the Hitchin equation is given by $$H=\begin{pmatrix}h_1&&&\\&h_2&&\\&&\ddots&\\&&&h_n\end{pmatrix}$$ 
where $h_i$ is the Hermitian metric on $L_i$ and under the isomorphism $det(E)\cong \mathcal{O}$, $det(H)$ is a constant metric $1$ on the line bundle $det(E)$ which is trivial;

II. the holomorphic structure on $E$ is given by the $\bar\partial$-operator \begin{eqnarray*}
\bar\partial_E=\begin{pmatrix}\bar\partial_1&\beta_{12}&\beta_{13}&\cdots&\beta_{1n}\\&\bar\partial_2&\beta_{23}&\cdots&\beta_{2n}\\&&\bar\partial_3&\cdots&\beta_{3n}\\&&&\ddots&\vdots\\&&&&\bar\partial_n\end{pmatrix}
\end{eqnarray*} where $\bar\partial_k$ are $\bar\partial$-operators defining the holomorphic structures on $L_k$, and $\beta_{ij}\in A^{0,1}(Hom(L_j,L_i))$;

III. the Higgs field is of the form \begin{eqnarray*}
&&\phi=\begin{pmatrix}a_{11}&a_{12}&a_{13}&\cdots&a_{1n}\\1_1&a_{22}&a_{23}&\cdots&a_{2n}\\&1_2&a_{33}&\cdots&a_{3n}\\&&\ddots&\ddots&\vdots\\&&&1_{n-1}&a_{nn}\end{pmatrix}
\end{eqnarray*} where $a_{ij}\in A^{1,0}(Hom(L_j,L_i))$ and $1_l$ is the constant map $1:K^{\frac{n+1-2l}{2}}\rightarrow K^{\frac{n+1-2(l+1)}{2}}\otimes K$ (the subscript $l$ will be useful later.)

\begin{rem} When the holomorphic decomposition is already orthogonal with respect to the harmonic metric $H$, then the new decomposition coincides with the old one. By the work of Baraglia \cite{Bar} and Collier \cite{BrianThesis}, there are several families of Higgs bundles in Hitchin component such that $H$ splits on the holomorphic splitting:\\
(1) cyclic case: $(0,\cdots,0,q_n)$;\\
(2) sub-cyclic case: $(0,\cdots,0,q_{n-1},0);$ and \\
(3) Fuchsian case: $(q_2,0,\cdots,0)$.  \\
In particular if $(E,\phi)$ is the base $n$-Fuchsian case of $\Sigma$, all the $\beta_{ij}, a_{ij}$ vanish. 
\end{rem}

\subsection{Hitchin equation in terms of orthogonal decomposition}
We are going to describe $\nabla^H$ and $\phi^{*_H}$ as follows. 

The Chern connection of the harmonic metric $H$ is 
\begin{eqnarray*}
\nabla^H=\begin{pmatrix}\nabla^{h_1}&\beta_{12}&\beta_{13}&\cdots&\beta_{1n}\\-\beta_{12}^*&\nabla^{h_2}&\beta_{23}&\cdots&\beta_{2n}\\-\beta_{13}^*&-\beta_{23}^*&\nabla^{h_3}&\cdots&\beta_{3n}\\\vdots&\vdots&\ddots&\ddots&\vdots\\-\beta_{1n}^*&-\beta_{2n}^*&-\beta_{2n}^*&\cdots&\nabla^{h_n}\end{pmatrix},
\end{eqnarray*}where $\nabla^{h_l}$ is Chern connection on each $L_l$, $\beta_{ij}^*$ is the Hermitian adjoint of $a_{ij}$, that is, $h_i(\beta_{ij}(e_j), e_i)=h_j(e_j, \beta_{ij}^*(e_i))$ for a local section $e_i, e_j$ of $L_i,L_j$ respectively. In local frame, $\beta_{ij}^*=\bar\beta_{ij}\cdot h_ih_j^{-1}.$

The Hermitian adjoint of the Higgs field with respect to $H$ is 
\begin{eqnarray*}
\phi^{*_H}=\begin{pmatrix}a_{11}^*&1_1^*&&&&\\a_{12}^*&a_{22}^*&1_2^*&&&\\a_{13}^*&a_{23}^*&a_{33}^*&\ddots&&\\\vdots&\vdots&\vdots&\ddots&1_{n-1}^*\\a_{1n}^*&a_{2n}^*&a_{3n}^*&\cdots&a_{nn}^*
\end{pmatrix},
\end{eqnarray*}
where $a_{ij}^*$ is the Hermitian adjoint of $a_{ij}$, that is, $h_i(a_{ij}(e_j), e_i)=h_j(e_j, a_{ij}^*(e_i))$ for a local section $e_i, e_j$ of $L_i,L_j$ respectively. Similarly, $1_l^*$ is the Hermitian adjoint of $1_l\in H^0(Hom(L_{l+1},L_l)\otimes K)$. In local frame, $a_{ij}^*=\bar a_{ij}\cdot h_ih_j^{-1}, 1_l^*=h_l^{-1}h_{l+1}.$

For $1\leq l\leq n$, the $(l,l)$-entry of the Hitchin equation $$\nabla^H\circ \nabla^H+[\phi,\phi^{*_H}]=0$$ is 
\begin{equation}\label{Firstequation}
F^{\nabla_{h_l}}-\sum\limits_{j=l+1}^n\beta_{lj}\wedge\beta^*_{lj}-\sum\limits_{j=1}^{l-1}\beta_{jl}^*\wedge\beta_{jl}+\sum\limits_{j=l+1}^na_{lj}\wedge a_{lj}^*+\sum\limits_{j=1}^{l-1}a_{lj}^*\wedge a_{lj}+1_{l-1}\wedge 1_{l-1}^*+1_l^*\wedge 1_l=0.
\end{equation}
In the case $l=1$, $1_{l-1}\wedge 1_{l-1}^*$ does not appear.

In a local coordinate $z$, choose a local holomorphic frame $s_l=dz^{\frac{n+1-2l}{2}}$ of $L_l$. Denote the local function $h_l(s_l,s_l)$ as $\hat h_l$. Then locally, $F^{\nabla h_l}=(\partial_{\bar z}\partial_z\log\hat h_l)d\bar z\wedge dz$. Getting rid of the $2$-form $d\bar z\wedge dz$, Equation (\ref{Firstequation}) is locally
\begin{equation}
\partial_{\bar z}\partial_z\log\hat h_l-\sum\limits_{j=l+1}^n(|\beta_{lj}|^2+|a_{lj}|^2)\hat h_l\hat h_j^{-1}+\sum\limits_{j=1}^{l-1}(|\beta_{jl}|^2+|a_{jl}|^2)\hat h_j\hat h_l^{-1}-\hat h_{l-1}^{-1}\hat h_l+\hat h_l^{-1}\hat h_{l+1}=0.
\end{equation}
In the case $l=1$, $\hat h_{l-1}^{-1}\hat h_l$ does not appear.

Let $g_0$ be the conformal hyperbolic metric on $\Sigma$ of constant curvature $-1$. By complexification, it also gives a Hermitian metric $h$ on $K^{-1}$. In local coordinate $z$, $g_0=2\hat g_0(dx^2+dy^2), h=\hat g_0dz\otimes d\bar z$. Locally $\hat g_0$ satisfies $$\partial_{\bar z}\partial_z\log \hat g_0=\hat g_0.$$ 

Since $h_l$ and $h^{-\frac{n+1-2l}{2}}$ are both two Hermitian metric on $L_l$, they differ by a positive function over the surface, i.e. $h_l=h^{-\frac{n+1-2l}{2}}\cdot e^{u_l}$, for some smooth function $u_l$ over $\Sigma$. Then locally, $u_l$ satisfies
\begin{eqnarray*}
\frac{1}{\hat g_0}\partial_{\bar z}\partial_z u_l-\sum\limits_{j=l+1}^n(|\beta_{lj}|^2+|a_{lj}|^2)\hat h_l\hat h_j^{-1}\hat g_0^{j-l-1}+\sum\limits_{j=1}^{l-1}(|\beta_{jl}|^2+|a_{jl}|^2)\hat h_j\hat h_l^{-1}\hat g_0^{j-l-1}\\
-\hat h_{l-1}^{-1}\hat h_l+\hat h_l^{-1}\hat h_{l+1}-\frac{n+1-2l}{2}=0.
\end{eqnarray*}

When $(E,\phi)$ corresponds to the base $n$-Fuchsian reresentation, all the $\beta_{jl}, a_{lj}$ vanish. Denote the Hermitian metric on $L_l$ as $\tilde h_l$. Then $\tilde h_l=h^{-\frac{n+1-2l}{2}}\cdot e^{\tilde u_l}$, for a constant $\tilde u_l$ satisfying $e^{-\tilde u_l+\tilde u_{l+1}}=\frac{(n-l)l}{2}$. The pullback metric of corresponding harmonic map $\tilde f$ is
\begin{eqnarray}\label{Fuchsian}\tilde f^*g_{G/K}=2Re(\frac{12}{n(n^2-1)}\cdot tr(\phi\phi^{*_H}))=2 Re(\frac{12}{n(n^2-1)}\cdot \sum_{l=1}^{n-1}\tilde h_l^{-1}\tilde h_{l+1})\nonumber\\
=2Re(\frac{12}{n(n^2-1)}\cdot \sum_{l=1}^{n-1}\frac{(n-l)l}{2}\cdot h)=2Re(h)=g_0.\end{eqnarray}
Hence in this case, the energy density $e(\tilde f)\equiv 1$.\\

Denote $\triangle_{g_0}=\frac{1}{\hat g_0}\partial_{\bar z}\partial_z$, which is globally well-defined on the surface. Denote $||\beta_{lj}||^2=|\beta_{lj}|^2\hat h_l\hat h_j^{-1}\hat g_0^{j-l-1}$ and $||a_{lj}||^2=|a_{lj}|^2\hat h_l\hat h_j^{-1}\hat g_0^{j-l-1}$, which are globally well-defined functions on the surface. So the equation of $u_l$ holds globally. 
\begin{equation}
\triangle_{g_0} u_l -\sum\limits_{j=l+1}^n(||\beta_{lj}||^2+||a_{lj}||^2)+\sum\limits_{j=1}^{l-1}(||\beta_{jl}||^2+||a_{jl}||^2)-e^{u_l-u_{l-1}}+e^{u_{l+1}-u_l}-\frac{n+1-2l}{2}=0.
\end{equation}
In the case $l=1$, $e^{u_l-u_{l-1}}$ does not appear.

Denote $z_l=u_l-\tilde u_l$, measuring how far $h_l$ is from $\tilde h_l$, the base $n$-Fuchsian case. So the function $z_l$ satisfies 
\begin{eqnarray}\label{firstmainequation}
\triangle_{g_0} z_l -\sum\limits_{j=l+1}^n(||\beta_{lj}||^2+||a_{lj}||^2)+\sum\limits_{j=1}^{l-1}(||\beta_{jl}||^2
+||a_{jl}||^2)-e^{z_l-z_{l-1}}\cdot \frac{(l-1)(n+1-l)}{2}\nonumber\\+e^{z_{l+1}-z_l}\cdot \frac{l(n-l)}{2}-\frac{n+1-2l}{2}=0.
\end{eqnarray}
In the case $l=1$, $e^{z_l-z_{l-1}}$ does not appear.

\section{Proof of main theorem}\label{proof}
We proceed to prove the main theorem in this section. We begin with Lemma \ref{estimate} estimating how far the Hermitian metric $h_l$ is from $\tilde h_l=g_0^{-\frac{n+1-2l}{2}}\cdot e^{\tilde u_l}$ (the base $n$-Fuchsian case) on the line bundle $L_l$, i.e. the function $z_l$ such that $h_l=\tilde h_l\cdot e^{z_l}$. 

Equation (\ref{firstmainequation}) of the function $z_l$ is difficult to directly estimate since the unknown terms $||\beta_{ij}||^2, ||a_{ij}||^2$ do not appear in the same sign. The main observation is that the sum of the first $k$ equations turns out to have the same sign surprisingly.
\begin{lem}\label{estimate}
Let $v_k=\sum\limits_{l=1}^kz_l$, if $\rho$ is not the base $n$-Fuchsian,
\begin{equation*}
v_k< 0, \quad\text{for $1\leq k\leq n-1$, and}\quad v_n\equiv 0.
\end{equation*}
\end{lem}
\begin{proof}
$v_n\equiv 0$ follows directly from both $H$ and $\tilde H$ are of unit determinant. 

Summing up Equation (\ref{firstmainequation}) over $1\leq l\leq k$, we obtain
\begin{eqnarray}\label{mainequation}
\triangle_{g_0} (\sum\limits_{l=1}^kz_l) -\sum\limits_{l=1}^k\sum\limits_{j=l+1}^n(||\beta_{lj}||^2+||a_{lj}||^2)+\sum\limits_{l+1}^k\sum\limits_{j=1}^{l-1}(||\beta_{jl}||^2+||a_{jl}||^2)\\
+e^{-z_k+z_{k+1}}\cdot\frac{k(n-k)}{2}-\sum\limits_{l=1}^k\frac{n+1-2l}{2}=0.\nonumber
\end{eqnarray}

Step 1: We first simplify the above Equation (\ref{mainequation}).

The following formula is key to us.
\begin{equation}\label{formula}
 -\sum\limits_{l=1}^k\sum\limits_{j=l+1}^n(||\beta_{lj}||^2+||a_{lj}||^2)+\sum\limits_{l=1}^k\sum\limits_{j=1}^{l-1}(||\beta_{jl}||^2+||a_{jl}||^2)=-\sum\limits_{l=1}^k\sum\limits_{j=k+1}^n(||\beta_{lj}||^2+||a_{jl}||^2).
\end{equation}
It holds because \begin{eqnarray*}
LHS&=&-\sum\limits_{l=1}^k\sum\limits_{j=l+1}^n(||\beta_{lj}||^2+||a_{lj}||^2)+\sum\limits_{j=1}^{k-1}\sum\limits_{l=j+1}^{k}(||\beta_{jl}||^2+||a_{jl}||^2)\\
 &=&-\sum\limits_{l=1}^k\sum\limits_{j=l+1}^n(||\beta_{lj}||^2+||a_{lj}||^2)+\sum\limits_{l=1}^{k-1}\sum\limits_{j=l+1}^{k}(||\beta_{lj}||^2+||a_{lj}||^2)\\
 &=&-\sum\limits_{l=1}^k\sum\limits_{j=k+1}^n(||\beta_{lj}||^2+||a_{lj}||^2)=RHS.
\end{eqnarray*}

Using Formula (\ref{formula}) and $\sum\limits_{l=1}^k\frac{n+1-2l}{2}=
\frac{(n-k)k}{2},$ Equation (\ref{mainequation}) becomes 
\begin{eqnarray*}
\triangle_{g_0} (\sum\limits_{l=1}^kz_l)-\sum\limits_{l=1}^k\sum\limits_{j=k+1}^n(||\beta_{lj}||^2+||a_{lj}||^2)+(e^{-z_k+z_{k+1}}-1)\cdot\frac{(n-k)k}{2}=0.
\end{eqnarray*}
Set $v_k=\sum_{l=1}^k z_k$ for $1\leq k\leq n$ and $v_0=0$. Noting that $e^{-z_k+z_{k+1}}=e^{v_{k-1}+v_{k+1}-2v_{k}}$ for $1\leq k\leq n-1$, we rewrite the above equation as
\begin{equation}\label{prekeyequation}
\triangle_{g_0} v_k-\sum\limits_{l=1}^k\sum\limits_{j=k+1}^n(||\beta_{lj}||^2+||a_{lj}||^2)+(e^{v_{k-1}+v_{k+1}-2v_k}-1)\cdot\frac{(n-k)k}{2}=0.
\end{equation} Therefore
\begin{equation}\label{keyequation}
\triangle_{g_0} v_k+(e^{v_{k-1}+v_{k+1}-2v_k}-1)\cdot\frac{(n-k)k}{2}\geq 0.
\end{equation} \\

Step 2: We show $v_k\leq 0$ for all $1\leq k\leq n$. 

Set $M=\max\limits_{1\leq k\leq n} \max\limits_{p\in \Sigma}v_k(p)$. Since $v_n\equiv 0$, $M\geq 0$. It suffices to show $M=0$. Suppose not, i.e. $M>0$. Assume $k_0$ is the largest integer such that $v_{k_0}$ achieves $M$ at some point. So $k_0\leq n-1$.
Then at some maximum point $p$ of $v_{k_0}$, $\triangle_{g_0}v_{k_0}(p)\leq 0$, using Equation (\ref{keyequation}), we obtain
\begin{eqnarray*}
&&e^{(v_{k_0-1}+v_{k_0+1}-2v_{k_0})(p)}-1\geq 0
\Longrightarrow v_{k_0-1}(p)+v_{k_0+1}(p)\geq  \max_{\Sigma}2v_{k_0}=2M.
\end{eqnarray*}
By the definition of $M$ being maximum, $v_{k_0-1}+v_{k_0+1}\leq 2M$, hence $v_{k_0-1}(p)=v_{k_0+1}(p)=M$. This contradicts the choice of $k_0$ being the largest integer such that $v_{k_0}$ achieves $M$ at some point.\\

Step 3: We show that $v_k<0$ for all $1\leq k\leq n-1$.

For each $k$, apply $v_{k-1}, v_{k+1}\leq 0$ to Equation (\ref{keyequation}), we obtain
\begin{eqnarray*}
\triangle_{g_0} v_k+(e^{-2v_k}-1)\cdot\frac{(n-k)k}{2}\geq 0.
\end{eqnarray*} 
Observe that the constant function $0$ satisfies $$\triangle_{g_0} (0)+(e^{-2\times 0}-1)\frac{(n-k)k}{2}= 0.$$ By the strong maximum principle (see \cite{Jost}, page 43), $v_k<0$ or $v_k\equiv 0$. In fact, the latter case cannot happen for any $1\leq k\leq n-1$. Suppose $v_{k_0}\equiv 0$, Equation (\ref{keyequation}) for $v_{k_0}$ implies $v_{k_0-1}=v_{k_0+1}\equiv 0$. Repeating this process, we obtain that for all $k$, $v_k\equiv 0$. Plugging in $v_k\equiv 0$ to Equation (\ref{prekeyequation}) for all $k$, we have that the $\beta_{lj},a_{lj}$ must vanish. Then the Higgs bundle is the base $n$-Fuchsian case, which cannot happen by assumption.
\end{proof}

Now we are ready to show the main theorem.
\begin{thm} \label{Maintheorem}
Let $\rho$ be a Hitchin representation for $PSL(n,\mathbb R)$, $g_0$ be a hyperbolic metric on $S$, and $f$ be the unique $\rho$-equivariant harmonic map from $(\widetilde S, \widetilde g_0)$ to $SL(n,\mathbb R)/SO(n)$. Then its energy density $e(f)$ satisfies
$$e(f)\geq 1.$$
Moreover, equality holds at one point only if $e(f)\equiv 1$ in which case $\rho$ is the base $n$-Fuchsian representation of $(S,g_0)$.  

\end{thm}
\begin{proof}、
Firstly, in the base $n$-Fuchsian representation case, we have $e(f)\equiv 1$ from Equation (\ref{Fuchsian}). Now suppose $\rho$ is not the base $n$-Fuchsian, we have
\begin{eqnarray*}
\text{tr}(\phi\phi^{*_H})=(\sum_{k=1}^{n-1}h_k^{-1}h_{k+1})+\sum\limits_{l=1}^n\sum\limits_{j=1}^{l}||a_{jl}||^2\geq \sum_{k=1}^{n-1}h_k^{-1}h_{k+1}=\sum_{k=1}^{n-1}\tilde h_k^{-1}\tilde h_{k+1}\cdot e^{-z_k+z_{k+1}}.
\end{eqnarray*}
Recall in the case of the base $n$-Fuchsian representation, $\tilde h_k^{-1}\tilde h_{k+1}=\frac{k(n-k)}{2}h$. So \begin{eqnarray*}
\text{tr}(\phi\phi^{*_H})/h&=&\sum\limits_{k=1}^{n-1}[\frac{k(n-k)}{2}\cdot e^{-z_k+z_{k+1}}]\\
&\geq& [\sum\limits_{k=1}^{n-1}\frac{k(n-k)}{2}]\cdot \Big \{e^{\sum\limits_{k=1}^{n-1}\frac{k(n-k)}{2}(-z_k+z_{k+1})}\Big\}^{\frac{1}{\sum\limits_{k=1}^{n-1}\frac{k(n-k)}{2}}}\\
&=&\frac{n(n^2-1)}{12}\cdot (e^{\sum\limits_{k=1}^{n-1}\frac{k(n-k)}{2}(-z_k+z_{k+1})})^{\frac{12}{n(n^2-1)}}\\
&=&\frac{n(n^2-1)}{12}\cdot (e^{-\sum\limits_{k=1}^{n-1}\frac{n+1-2k}{2}z_k})^{\frac{12}{n(n^2-1)}}.\\
&&\text{using $z_n=-z_1-z_2-\cdots-z_{n-1}$}\\
&=&\frac{n(n^2-1)}{12}\cdot(e^{-[(n-1)z_1+(n-2)z_2+\cdots+2z_{n-2}+z_{n-1}]})^{\frac{12}{n(n^2-1)}}\\
&=&\frac{n(n^2-1)}{12}\cdot(e^{-\sum\limits_{k=1}^{n-1}z_k}\cdot e^{-\sum\limits_{k=1}^{n-2}z_k}\cdot \cdots\cdot e^{-\sum\limits_{k=1}^{2}z_k} \cdot e^{-z_1})^{\frac{12}{n(n^2-1)}}\\
&>&\frac{n(n^2-1)}{12},\quad \text{by Lemma \ref{estimate}, $v_k=\sum_{l=1}^kz_k<0$}.\\
\end{eqnarray*}
Hence $e(f)=2Re(\frac{12}{n(n^2-1)}\cdot \text{tr}(\phi\phi^{*_H}))/g_0>2Re(h)/g_0=1$.
\end{proof}

\section{Further questions}\label{question}
There is a natural $\mathbb{C}^*$-action on the moduli space $M_{Higgs}$ of $SL(n,\mathbb{C})$-Higgs bundles given by
\begin{eqnarray*}
\mathbb C^*\times \mathcal{M}_{Higgs}&\rightarrow& \mathcal{M}_{Higgs}\\
t\cdot (E,\phi)&=&(E,t\phi).\end{eqnarray*} Hitchin \cite{Hitchin92} showed that along the $\mathbb{C}^*$-flow, the energy is monotonically increasing as $|t|$ increases. From integral monotonicity to pointwise monotonicity, we ask the following natural question.

\begin{conj}\label{flow}
Along the $\mathbb{C}^*$-flow on the moduli space $\mathcal{M}_{Higgs}$, the energy density of the corresponding harmonic maps is monotonically increasing as $|t|$ increases. 
\end{conj}

For any Higgs bundle $(E,\phi)$ in the Hitchin component, as $t\rightarrow 0$, the limit $\lim_{t\rightarrow 0}t\cdot (E,\phi)$ corresponds to the base $n$-Fuchsian case. Therefore, the conjecture is a natural generalisation of Theorem \ref{maintheorem}. For cyclic Higgs bundles, Conjecture \ref{flow} is shown by the author and Dai in \cite{DL2}. \\

Notice that the $\mathbb C^*$ always take elements in one Hitchin fiber to elements in distinct Hitchin fiber unless the element is based at the origin. If instead, we stay at in the same Hitchin fiber, Hitchin section is expected to be the maximum in terms of the energy density.
\begin{conj}\label{main} 
Let $(\tilde{E},\tilde{\phi})$ be a Higgs bundle in the Hitchin component and $(E,\phi)$ be a distinct polystable $SL(n,\mathbb{C})$-Higgs bundle in the same Hitchin fiber. Then the corresponding harmonic maps $f, \tilde f$ satisfy $e(f)<e(\tilde f)$ and hence $f^*g_{G/K}<\tilde f^*g_{G/K}.$ 

As a result, the energy satisfies $E(f)<E(\tilde f)$.
\end{conj}

In the $SL(2,\mathbb C)$ case, this is shown by Deroin-Tholozan \cite{DominationFuchsian}. 

In an upcoming paper \cite{Li}, the author will generalize the result of Deroin and Tholozan and show that in the Hitchin fiber at $(q_2,0,\cdots,0)$, the Hitchin section achieves the maximum of the energy density. In this case, it will imply the domination of the classical translation length spectrum of the representation by an $n$-Fuchsian representation.

\end{document}